\documentclass{amsart}
\usepackage{graphicx,epsfig,color}
\usepackage{amsmath,amsthm,amssymb}
\usepackage{framed}
\usepackage{dsfont}
\usepackage{amsfonts}
\usepackage{amsmath}
\usepackage{epsfig}
\usepackage{graphicx}
\usepackage{amssymb}

\newtheorem{theoreme}{Theorem}[section]
\newtheorem{proposition}[theoreme]{Proposition}
\newtheorem{lemma}[theoreme]{Lemma}
\newtheorem{corollary}[theoreme]{Corollary}

\theoremstyle{definition}

\newtheorem{remarque}[theoreme]{Remark}

\newcommand {\nc}   {\newcommand}
\nc {\be}   {\begin{equation}} \nc {\ee}   {\end{equation}} \nc
{\beq}  {\begin{eqnarray}} \nc {\eeq}  {\end{eqnarray}} \nc {\beqs}
{\begin{eqnarray*}} \nc {\eeqs} {\end{eqnarray*}}
\def\edc{\end{document}}
\numberwithin{equation}{section}
\begin{document}

\author{Nasab Yassine}
\address{Universit\'e de Bretagne Occidentale
\\
LMBA\\
CNRS UMR 6205\\
Institut des sciences et Techniques\\
29238 Brest Cedex 3, France}
\email{nasab.yassine@univ-brest.fr}

\title[Quantitative recurrence of some dynamical systems with an infinite measure in dimension one]{Quantitative recurrence of some dynamical systems with an infinite measure in dimension one}
\begin{abstract}
We are interested in the asymptotic behaviour of the first return time of the orbits of a dynamical system into a small neighbourhood of their starting points. We study this quantity in the context of dynamical systems preserving an infinite measure. More precisely, we consider the case of $\mathbb{Z}$-extensions of subshifts of finite type. We also consider a toy probabilistic model to enlight the strategy of our proofs.
\end{abstract}

\maketitle

\section{Introduction}

The quantitative recurrence properties of dynamical systems preserving a probability measure have been studied by many authors since the work of Hirata \cite{hirata}. Some properties are defined by estimating the first return time of a dynamical system into a small neighbourhood of its starting point. Results in this concern have been described in \cite{saussol}, let us mention works in this situation \cite{abadigalves, saussolrapidly}.
This question has been less investigated in the context of dynamical systems preserving an infinite measure. In \cite{xavier}, Bressaud and Zweim\"{u}ller have established first results of quantitative recurrence for piecewise affine maps of the interval with infinite measure. The case of $\mathbb{Z}^2$-extension of mixing subshifts of finite type has been investigated in \cite{penesaussol}. Results have been also established for random walks on the line \cite{pzs}, for billiards in the plane \cite{penesaussolbilliard} and for null-recurrent Markov maps in \cite{pzs2}. \\
A measure-preserving dynamical system is given by $(X,\mathcal{B},\mu,T)$ where $(X,\mathcal{B})$ is a measurable set, $\mu$ is a finite or $\sigma$-finite positive measure and $T:X \rightarrow X$ is a measurable transformation preserving the measure $\mu$ (i.e. $\mu(T^{-1}A)=\mu(A)$, for every $A\in\mathcal{B})$. We are interested in the case where $\mu$ is $\sigma$-finite. We assume that $X$ is endowed with some metric $d_{X}$ and that $\mathcal{B}$ contains the open balls $B(x,r)$ of $X$. Our interest is in the first time the orbit comes back close to its initial position. For every $y\in X$, we define the first return time $\tau_{\epsilon}$ of the orbit of $y$ in the ball $B(y,\epsilon)$ as:
\begin{equation*}
\tau_{\epsilon}(y):=\inf\{n\geq1: T^{n}(y)\in B(y,\epsilon)\}\in \mathbb{N}\cup\{+\infty\}.
\end{equation*}
We consider conservative dynamical systems, that is dynamical systems for which the conclusion of the poincar\'{e} theorem is satisfied. This ensures that, for every $\epsilon >0$, $\tau_{\epsilon}<\infty$, $\mu$ almost everywhere. The main goal of this article is to study the behavior of $\tau_{\epsilon}$ as $\epsilon\rightarrow 0$. A classical example of dynamical systems preserving an infinite measure is given by $\mathbb{Z}$-extensions of a probability-preseving dynamical system. Given a probability-preserving dynamical system $(\bar{X},\bar{\mathcal{B}},\nu,\bar{T})$ and a measurable function $\varphi:\bar{X}\rightarrow \mathbb{Z}$, we construct the $\mathbb{Z}$-extension $(X,\mathcal{B},\mu,T)$ of $(\bar{X},\bar{\mathcal{B}},\nu,\bar{T})$ by setting $X:=\bar{X}\times\mathbb{Z}$, $\mathcal{B}:=\bar{\mathcal{B}}\otimes\mathcal{P}(\mathbb{Z})$, $\mu:= \nu\otimes\sum_{l\in\mathbb{Z}}\delta_{l}$ and $T(x,l)=(\bar{T}(x),l+\varphi(x))$. We endow $X$ with the product metric given by $d_{X}((x,l),(x^{'}, l^{'})):=\max\{d_{\bar{X}}(x,x^{'}), \mid l-l^{'}\mid\}$. Hence $T^{n}(x,l)=(\bar{T}^{n}x, l+S_{n}\varphi(x))$, where $S_{n}\varphi$ is the ergodic sum $S_{n}\varphi:=\sum_{k=0}^{n-1}\varphi\circ \bar{T}^{k}$. Therefore, for $\epsilon$ small enough,
\begin{equation*}
T^{n}(x,l)\in B((x,l),\epsilon)\Longleftrightarrow\bar{T}^{n}(x)\in B_{\bar{X}}(x,\epsilon)\textit{ and } S_{n}\varphi(x)=0.
\end{equation*} 
Our main results concern the case when $(\bar{X},\bar{\mathcal{B}},\nu,\bar{T})$ is a mixing subshift of finite type (see Section 3 for precise definition), which are classical dynamical systems used to model a wide class of dynamical systems such as geodesic flows in negative curvature, etc.\\
Consider $(\bar{X},\bar{\mathcal{B}},\nu,\bar{T})$ a mixing subshift of finite type and $\nu$ a Gibbs measure associated to a H\"{o}lder continuous potential. Moreover we have a $\nu$-centered H\"{o}lder continuous function $\varphi$. Then we get
\begin{equation}\label{int1}\
\underset{\epsilon\rightarrow0}{\lim}\frac{\log\tau_{\epsilon}}{\log\epsilon}=-2d,
\end{equation}
$\mu$-almost everywhere, where $d$ is the Hausdorff dimension of $\nu$. Moreover the following convergence holds in distribution with respect to any probability measure absolutely continuous with respect to $\mu$:
\begin{equation}\label{int2}
\mu(B(.,\epsilon))\sqrt{\tau_{\epsilon}(.)}\underset{\epsilon\rightarrow0}{\longrightarrow}\frac{\mathcal{E}}{|\mathcal{N}|},
\end{equation}
where $\mathcal{E}$ and $\mathcal{N}$ are two independent random variables with respective exponential distribution of mean 1 and standard normal distribution (see Theorem \ref{theorem1.1} and Theorem \ref{theorem1.2} for precise statements).\\
  Roughly speaking the strategy of our proof is that there is a large scale (corresponding to $S_{n}\varphi(x)$) and a small scale (corresponding to $\bar{T}^{n}(x)$), which behave independently assymptotically.
  To enlight this strategy, we start out this paper with the study of the toy probabilistic model $(Y_{n},S_{n})$, where $(S_{n})_{n}$ is the simple symmetric random walk and $(Y_{n})_{n}$ is a sequence of independent random variables, with uniform distribution on $(0,1)^{d}$ and where $S_{n}$ and $Y_{n}$ are independent. For this simple model, we obtain the same results. More precisely, we prove that \eqref{int1} holds almost surely and that \eqref{int2} holds in distribution. 
\section{ toy probabilistic model}
 Let $d\in\mathbb{N}$. In this section, we give a real random walk $(M_{n})_{n\geq0}$ with values in $\mathbb{R}\times ]0,1[^{d-1}\subset \mathbb{R}^{d}$.
\subsection{Description of the model and statement of the results}
The random process $M_{n}$ is given by $M_{n}=(S_{n},0)+Y_{n}$. $(S_{n})_{n\geq0}$ and $(Y_{n})_{n\geq0}$ are independent such that:
\begin{itemize}
\item $Y_{n}$ is uniformly distributed on $(0,1)^{d}$.
\item $S_{n}$ is the simple symmetric random walk on $\mathbb{Z}$ given by $S_{0}=0$, i.e. $S_{n}=\sum_{k=1}^{n}X_{k}$, where $(X_{k})_{k}$ is a sequence of independent random variables such that: $\mathbb{P}(X_{k}=1)=\mathbb{P}(X_{k}=-1)=1/2$.
\end{itemize}
We want to study the asymptotic behavior, as $\epsilon$ goes to $0$, of $\tau_{\epsilon}$ for the metric associated to some norm on $\mathbb{R}^{d}$. Let $c$ be the Lebesgue measure of the unit ball in $\mathbb{R}^{d}$. We will prove the following:
\begin{theoreme}\label{theorem1.1}
Almost surely, $\frac{\log{\tau_{\epsilon}}}{-\log{\epsilon}}$ converges to $2d$ as $\epsilon$ goes to $0$.
\end{theoreme}
For this constant $c>0$, we have the following result:
\begin{theoreme}\label{theorem1.2}
The sequence of random variables $((c\epsilon^{d})\sqrt{\tau_{\epsilon}})_{\epsilon}$ converges in distribution to $\frac{\mathcal{E}}{\mathcal{N}}$, where $\mathcal{E}$ and $\mathcal{N}$ are two independent random variables, $\mathcal{E}$ having an exponential distribution of mean $1$ and $\mathcal{N}$ having a standard Gaussian distribution.
\end{theoreme}
\subsection{Proof of the pointwise convergence of the recurrence rate to the dimension}
 $M_{0}$ is in $)0;1(^{d}$, let $\epsilon$ so small that $B(M_{0};\epsilon)$ is contained in $)0,1(^{d}$. Note that $Leb(B(x,\epsilon))= c \epsilon^{d}$.\\
 We define for any $p\geq0$ the $p^{th}$ return time $R_{p}$ of $(M_{n})_{n}$in $)0;1(^{d}$, setting $R_{0}=0$, by induction :
\begin{equation*}
R_{p+1}:=\inf \big\{m>R_{p} : S_{m}=0\big\}.\\
\end{equation*}
We have the relation:
\begin{equation*}
\tau_{\epsilon}=R_{T_{\epsilon}}\text{ with } T_{\epsilon}:=\min\{l\geq{1} : Y_{R_{l}}\in B(Y_{0},\epsilon)\}
\end{equation*}
We will study the asymptotic behavior of the random variables $R_{n}$ and $T_{\epsilon}$ and use the relation between them to prove Theorem \ref{theorem1.1}.
\subsubsection{Study the return of the random variable $R_{n}$.}
\begin{proposition}[Feller \cite{feller}] There exists $C>0$ such that:
\begin{equation}\label{1}
\mathbb{P}(R_{1}>s)\sim \frac{C}{\sqrt{s}}, \quad \mbox{as $s\rightarrow  \infty$}
\end{equation}
\end{proposition}
\begin{remarque}
Due to the strong Markov property, the delays $U_{p}:=R_{p}-R_{p-1}$ between successive return times are independent and identically distributed.
\end{remarque}\label{2} 
\begin{lemma}\label{lemmaRn}
Almost surely, $\frac{\log{\sqrt{R_{n}}}}{\log{n}}$ converges to 1 as $n$ goes to $\infty$.
\end{lemma}
\proof
The proof of the lemma directly holds, once the following inequality is proved:
\begin{equation*}
\forall \alpha\in(0,1),\exists n_{0},\forall n\geq n_{0},\quad n^{1-\alpha}\leq{\sqrt{R_{n}}}\leq{n^{1+\alpha}}
\end{equation*}
Let $\alpha\in(0,1)$, by independence (using Remark \ref{2}), we have:
\begin{equation}\label{eqRn}
\mathbb{P}(\sqrt{R_{n}}\leq{n^{1-\alpha}})\leq\mathbb{P}(\forall{p}\leq{n},\sqrt{R_{p}-R_{p-1}}\leq{n^{1-\alpha}})\\
\leq\mathbb{P}\left(\sqrt{R_{1}}\leq{n^{1-\alpha}}\right)^{n}.
\end{equation}
Due to the asymptotic formula given in Proposition \ref{1}, for $n$ sufficiently large
\begin{equation*}
\mathbb{P}(\sqrt{R_{1}}\leq{n^{1-\alpha}})^{n}\leq\left(1-\frac{C}{2n^{1-\alpha}}\right)^{n}\leq\exp{\left(-C\frac{n^{\alpha}}{2}\right)}.
\end{equation*}
This allows us to get the first inequality of \eqref{eqRn} by using the Borel Cantelli lemma. Again, using proposition~\ref{1}, we have $P\left(R_{1}^{\frac{1}{2+\alpha}}>s \right)\leq\frac{C^{'}}{s^{1+\frac{\alpha}{2}}}$ for some $C^{'}>0$, implying obviously that $\mathbb{E}\left(R_{1}^{\frac{1}{2+2\alpha}}\right)<\infty$.\\
Note that one can see,
\begin{equation*}
R_{n}=\sum_{i=1}^{n}U_{i}\leq n^{2+2\alpha}\left(\frac{1}{n}\sum_{i=1}^{n}U_{i}^{\frac{1}{2+2\alpha}}\right)^{2+2\alpha}.
\end{equation*}
But $\frac{1}{n}\sum_{i=1}^{n}U_{i}^{\frac{1}{2+2\alpha}}$ converges almost surely to $\mathbb{E}\left(R_{1}^{\frac{1}{2+2\alpha}}\right)<\infty$ due to the strong law of large numbers. Hence $R_{n}= O(n^{2+2\alpha})$ almost surely, from which we get the second inequality.
\subsubsection{Study the return of the random variable $T_{\epsilon}$}
In this subsection the asymptotic behavior of the random variable $T_{\epsilon}$ is illustrated in the following lemma.
\begin{lemma}\label{lemmaTs}
Almost surely, $\frac{\log{T_{\epsilon}}}{-\log{\epsilon}} \rightarrow $ $d$  as  $\epsilon \rightarrow 0$.
\end{lemma}
\begin{proof}
Given $Y_{0}$, let $\epsilon >0$ be such that $B(Y_{0}, \epsilon)\subset(0,1)^{d}$. The random variable $T_{\epsilon}$ has a geometric distribution with parameter $\lambda_{\epsilon}:=c{\epsilon}^{d}$.\\
For any $\alpha>0$, a simple decomposition gives:
\begin{equation*}
\mathbb{P}\left(\bigg|\frac{\log{T_{\epsilon}}}{-\log{\epsilon}}-d\bigg|>\alpha\right)=\mathbb{P}\left(T_{\epsilon}>\epsilon^{-d-\alpha})+\mathbb{P}(T_{\epsilon}<\epsilon^{-d+\alpha}\right).
\end{equation*}
The first term is handled by the Markov inequality:
\begin{equation*}
\mathbb{P}(T_{\epsilon}>\epsilon^{-d-\alpha}\mid Y_{0})\leq \epsilon^{\alpha}\frac{{\epsilon}^{d}}{\lambda_{\epsilon}}=O(\epsilon^{\alpha}).
\end{equation*}
While the second term using the geometric distribution:
\begin{eqnarray*}
\mathbb{P}(T_{\epsilon}<\epsilon^{-d+\alpha})
&=&1-(1-c\epsilon^d)^{\epsilon^{-d+\alpha}}\\
&\leq&1-\exp[\epsilon^{-d+\alpha}\log(1-c\epsilon^d)]\\
&\leq&(-\epsilon)^{-d+\alpha}\log(1-c\epsilon^d)\\
&=&O(\epsilon^{\alpha}).
\end{eqnarray*}
Let us define $\epsilon_{n}:=n^{\frac{-2}{\alpha}}$. Thus $(\epsilon_{n})_{n\geq1}$ is a decreasing sequence of real numbers, and $T_{\epsilon}$ is monotone in $\epsilon$, so that:
\begin{equation*}
\sum_{n\geq1}\mathbb{P}(|\frac{\log{T_{\epsilon_{n}}}}{-\log{\epsilon_{n}}}-d|>\alpha)<+\infty.
\end{equation*}
 According to Borel Cantelli lemma $\frac{\log{T_{\epsilon_{n}}}}{-\log{\epsilon_{n}}} \rightarrow d$ almost surely as $n \rightarrow +\infty$.\\
Hence the proof follows since $\underset{n\rightarrow +\infty}{\lim}\epsilon_{n}=0$ and $\underset{n\rightarrow +\infty}{\lim}\frac{\epsilon_{n}}{\epsilon_{n+1}}=1$.
\end{proof}
\emph{\textbf{Proof of Theorem \ref{theorem1.1}}} 
The theorem follows from the two previous lemmas \ref{lemmaRn} and \ref{lemmaTs}, since:
\begin{equation*}
\frac{\log{\sqrt{\tau_{\epsilon}}}}{-\log{\epsilon}}=\frac{\log{\sqrt{R_{T_{\epsilon}}}}}{\log{T_{\epsilon}}}\frac{\log{T_{\epsilon}}}{-\log{\epsilon}}\rightarrow 1\times d=d \quad \text{a.s.}
\end{equation*}
Hence, we get: 
\begin{equation*}
\frac{\log{\tau_{\epsilon}}}{-\log{\epsilon}} \rightarrow 2d \text{ as } \epsilon \rightarrow 0 \quad \text{a.s.}
\end{equation*}

\subsection{Proof of the convergence in distribution of the rescaled return time.}
\begin{proposition}\label{prop1.1.3}
The sequence of random variables $(\frac{R_{n}}{n^{2}})_{n}$ converges in distribution to $\mathcal{N}^{-2}$ where $\mathcal{N}$ is a standard Gaussian random variable.
\end{proposition}
The proof of this proposition follows from the two following successive lemmas; the proof of which is straightforward and is omitted.
\begin{lemma}
$\sum_{n\geq0}\mathbb{P}(S_{2n}=0)s^{2n}=\frac{1}{\sqrt{1-s^{2}}}$ and $\mathbb{P}(S_{2n}=0)~\frac{1}{\sqrt{\pi n}}$.
\end{lemma}
Note that $\mathbb{P}(S_{2n}=0)=\sum_{k=0}\mathbb{P}(S_{k}=0)\mathbb{P}(R_{1}=2n-2k)$.
Hence, $\sum_{n>1}\mathbb{P}(S_{2n}=0)s^{2n}=\left(\sum_{n\geq0}\mathbb{P}(S_{2n}=0)s^{2n}\right)\mathbb{E}(s^{R_{1}})$. And so $\mathbb{E}\left[s^{R_{1}}\right]=1-\sqrt{1-s^{2}}$.
\begin{lemma}\label{lemma moment}
The moment generating function of $\mathcal{N}^{-2}$ is $\mathbb{E}\left[e^{-t\mathcal{N}^{-2}}\right]=e^{-\sqrt{2t}}$, $\forall t\geq0$, where $\mathcal{N}$ is standard Gaussian random variable.
\end{lemma}
\begin{proof}[Proof of Proposition \ref{prop1.1.3}.] Knowing that $R_{1}, (R_{2}-R_{1}), ..., (R_{n}-R_{n-1})$ are i.i.d., and the fact that $\mathbb{E}\left[s^{R_{1}}\right]=1-\sqrt{1-s^{2}}$, we get:
\begin{equation*}
\mathbb{E}[e^{-\frac{t}{n^{2}}R_{n}}]=\left(\mathbb{E}[e^{-\frac{t}{n^{2}}R_{1}}]\right)^{n}=\left[1-\sqrt{1-e^{-2\frac{t}{n^{2}}}}\right]^{n}
\end{equation*}
and from Lemma \ref{lemma moment}, we have:
\begin{equation*}
\forall t\geq0, \quad \lim_{n\rightarrow\infty}\mathbb{E}\left[e^{-\frac{t}{n^{2}}R_{n}}\right]=e^{-\sqrt{2t}}=\mathbb{E}[e^{-t\mathcal{N}^{-2}}].
\end{equation*}
Hence, $(\frac{R_{n}}{n^{2}})_{n}$ converges in distribution to $\mathcal{N}^{-2}$.
\end{proof}
\begin{lemma}\label{lemma2}
$(\lambda_{\epsilon}T_{\epsilon})_{\epsilon}$ converges in distribution to an exponential random variable $\mathcal{E}$ of mean 1. 
\end{lemma}
\begin{proof}
Given $Y_{0}$, $T_{\epsilon}$ has a geometric distribution of parameter $\lambda_{\epsilon}=\lambda(B(Y_0,\epsilon))$. Let $t>0$,
\begin{equation*}
\mathbb{P}(\lambda_{\epsilon} T_{\epsilon}\leq t\mid Y_0)
=\sum_{n=1}^{\left\lfloor\frac{t}{\lambda_{\epsilon}}\right\rfloor}\lambda_{\epsilon}(1-\lambda_{\epsilon})^{n-1}
=1-\exp\left({\left\lfloor\frac{t}{\lambda_{\epsilon}}\right\rfloor}\log(1-\lambda_{\epsilon})\right),
\end{equation*}
it follows that, for $\mathcal{E}$ a random variable which follows $\exp(1)$,
\begin{equation*}
\underset{\epsilon\rightarrow0}{\lim}\ \mathbb{P}(\lambda_{\epsilon} T_{\epsilon}\leq t\mid Y_0)=1-e^{-t}=\mathbb{P}(\mathcal{E}\leq t),\quad a.s.
\end{equation*}
\end{proof}
\begin{proof}[Proof of Theorem \ref{theorem1.2}] 
 Let us prove that the family of couples $\left(\lambda_{\epsilon}T_{\epsilon},\frac{R_{T_{\epsilon}}}{T_{\epsilon}}\right)_{\epsilon>0}$ converges in distribution, as $\epsilon \rightarrow0$, to $(\mathcal{E},\mathcal{N}^{-2})$, where $\mathcal{E}$ and $\mathcal{N}^{-2}$ are assumed to be as above and independent.\\
Let $s>0$ and $t\in \mathbb{R}$ ,then using the independence of $(T_{\epsilon})_{\epsilon}$ and $(R_{n})_{n}$, we get:
\begin{eqnarray*}
\bigg|\mathbb{P}\left(\lambda_{\epsilon}T_{\epsilon}>s,\frac{R_{T_{\epsilon}}}{T^2_{\epsilon}}>t\right)-\mathbb{P}\left(\lambda_{\epsilon}T_{\epsilon}>s)\mathbb{P}(\mathcal{N}^{-2}>t\right)\bigg|
&\leq&\sum_{n>\frac{s}{\lambda_{\epsilon}}}\lambda_{\epsilon}(1-\lambda_{\epsilon})^{n-1}\bigg|\mathbb{P}\left(\frac{R_{n}}{n^{2}}>t\right)-\mathbb{P}\left(\mathcal{N}^{-2}>t\right)\bigg|\\
&\leq&\underset{n>\frac{s}{c\epsilon^{d}}}{\sup}\bigg|\mathbb{P}\left(\frac{R_{n}}{n^{2}}>t\right)-\mathbb{P}(\mathcal{N}^{-2}>t)\big|
\end{eqnarray*}
This latter goes to 0 as $\epsilon$ goes to 0, due to Proposition \ref{prop1.1.3}. Moreover by Lemma \ref{lemma2}, $\mathbb{P}(\lambda_{\epsilon}T_{\epsilon}>s)\rightarrow \mathbb{P}(\mathcal{E}>s)$ as $\epsilon \rightarrow 0$, hence:
\begin{equation*}
\forall s>0, \forall t \quad \underset{\epsilon\rightarrow0}{\lim}\ \mathbb{P}(\lambda_{\epsilon}T_{\epsilon}>s,\frac{R_{T^{2}_{\epsilon}}}{T_{\epsilon}}>t)-\mathbb{P}(\mathcal{E}>s,\mathcal{N}^{-2}>t)=0.
\end{equation*}
This proves that the couple $\left(\lambda_{\epsilon}T_{\epsilon},\frac{R_{T_{\epsilon}}}{T_{\epsilon}^{2}}\right)_{\epsilon>0}$ converges in distribution, as $\epsilon$ goes to 0, to $(\mathcal{E},\mathcal{N}^{-2})$.\\
Knowing that $\tau_{\epsilon}=R_{T_{\epsilon}}$, we thus find that:
\begin{equation*}
(c\epsilon^{d})^{2}\tau_{\epsilon}
=\left(\frac{c\epsilon^{d}}{\lambda_{\epsilon}}\right)^{2}\lambda_{\epsilon}^{2}T^{2}_{\epsilon}\frac{R_{T_{\epsilon}}}{T_{\epsilon}^{2}}\\
\end{equation*}
Since $(x,y)\mapsto x^{2}y$ is continuous,
$\lambda_{\epsilon}^{2}T^{2}_{\epsilon}\frac{R_{T_{\epsilon}}}{T_{\epsilon}^{2}}\overset{d}{\longrightarrow}\mathcal{E}^{2}\mathcal{N}^{-2} \text{ as } \epsilon \rightarrow0$.
Observe that $\left(\frac{c\epsilon^{d}}{\lambda_{\epsilon}}\right)^{2}\overset{a.s.}{\longrightarrow}1$, hence by Slutzky's Lemma, we end up with :
\begin{equation*}
(c\epsilon^{d})^{2}\tau_{\epsilon}\overset{d}{\rightarrow}\mathcal{E}^{2}\mathcal{N}^{-2},\quad \text{ as } \epsilon \rightarrow0.
\end{equation*}
\end{proof}
\section{$\mathbb{Z}$-extension of a mixing subshift of finite type}
Let $\mathcal{A}$ be a finite set, called the alphabet, and let $M$ be a matrix indexed by $\mathcal{A}\times\mathcal{A}$ with 0-1 entries. We suppose that there exists a positive integer $n_{0}$ such that each component of $M^{n_{0}}$ is non zero. The subshift of finite type with alphabet $\mathcal{A}$ and transition matrix $M$ is $(\Sigma, \theta)$, with
\begin{equation*}
\Sigma := \{w:=(w_{n})_{n\in \mathbb{Z}} : \forall n \in \mathbb{Z}\text{ , } M(w_{n},w_{n+1})=1 \}
\end{equation*}
together with the metric $d(w,w^{'}):= e^{-m}$, where $m$ is the greatest integer such that $w_{i}=w_{i}^{'}$ whenever $|i|<m$, and the shift $\theta:\Sigma \rightarrow \Sigma$, $\theta((w_{n})_{n\in \mathbb{Z}})=(w_{n+1})_{n\in \mathbb{Z}}$.
 Let $\nu$ be the Gibbs measure on $\Sigma$ associated to some H\"{o}lder continuous potential $h$, and denote by $\sigma^{2}_{h}$ the asymptotic variance of $h$ under the measure $\nu$. Recall that $\sigma^{2}_{h}$ vanishes if and only if $h$ is cohomologous to a constant, and in this case $\nu$ is the unique measure of maximal entropy.\\
 For any function $f: \Sigma \rightarrow \mathbb{R}$ we denote by $S_{n}f:= \Sigma_{l=0}^{n-1}f\circ \theta^{l}$ its ergodic sum.
 Let us consider a H\"{o}lder continuous function $\varphi : \Sigma \rightarrow \mathbb{Z} $, such that $\int \varphi d\nu=0$. We consider the $\mathbb{Z}$-extension $F$ of the shift $\theta$ by $\varphi$. Recall that
\begin{eqnarray*}
F : \Sigma \times \mathbb{Z} &\rightarrow& \Sigma \times \mathbb{Z}\\
(x,m)&\rightarrow& (\theta x, m+\varphi(x)).
\end{eqnarray*}
Recall that $\Sigma\times\mathbb{Z}$ is endowed with distance $d_{0}((w,l),(w^{'},l^{'})):=\max\{d(w,w^{'}),\mid l-l^{'}\mid\}$. Note that, if $\epsilon<1$, for every $(w,l)\in \Sigma\times \mathbb{Z}$, we have $\mu \left(B_{\Sigma\times\mathbb{Z}}((w,l),\epsilon)\right)=\nu(B_{\Sigma}(w,\epsilon))$.
We want to know the time needed for a typical orbit starting at $(x,m)\in\Sigma \times \mathbb{Z}$ to return $\epsilon$-close to the initial point after iterations of the map $F$. By the translation invariance we can assume that the orbit starts in the cell $m=0$. Recall that
\begin{eqnarray*}
\tau_{\epsilon}(x)&=&\min\{n\geq1 : F^{n}(x,0)\in B(x, \epsilon)\times \{0\}\}\\
&=&\min\{n\geq1 : S_{n}\varphi(x)=0 \text{ and }  d(\theta^{n}x,x)<\epsilon \}.
\end{eqnarray*}
 We know that there exists a positive integer $m_{0}$ such that the function $\varphi$ is constant on each $m_{0}$-cylinders.\\
 Let us denote by $\sigma^{2}_{\varphi}$ the asymptotic variance of $\varphi$:
\begin{equation*}
\sigma^{2}_{\varphi}=\underset{n\rightarrow \infty}{\lim} \frac{1}{n}\mathbb{E}[(S_{n}\varphi)^2].
\end{equation*}
  We assume that $\sigma^{2}_{\varphi}\neq0$ (otherwise $(S_{n}\varphi)_{n}$ is bounded).
 We reinforce this by the following non-arithmeticity hypothesis on $\varphi$: We suppose that, for any $u\in[-\pi;\pi]\backslash\{0\}$ the only solutions $(\lambda,g)$, with $\lambda \in \mathbb{C}$ and $g : \Sigma \rightarrow \mathbb{C}$ measurable with $|g|=1$, of the functional equation
 \begin{equation}\label{nonarithmeticity}
 g\circ\theta \overset{-}{g}=\lambda e^{iu. \varphi}
 \end{equation}
 is the trivial one $\lambda=1$ and $g=const$. The fact that there is no non constant $g$ satisfying \eqref{nonarithmeticity} for $\lambda=1$ ensures that $\varphi$ is not a coboundary and so that $\sigma_{\varphi}^{2}\neq0$.
  The fact that there exists $(\lambda,g)$ satisfying \eqref{nonarithmeticity} with $\lambda\neq1$ would mean that the range of $S_{n}\varphi$ is essentially contained in a sub-lattice of $\mathbb{Z}$; in this case we could just do a change of basis and apply our result to the new reduced $\mathbb{Z}$-extension. We emphasize that this non-arithmeticity condition is equivalent to the fact that all the circle extensions $T_{u}$ defined by $T_{u}(x,t)=(\theta(x), t + u.\varphi(x))$ are weakly mixing for $u \in [-\pi;\pi]\backslash\{0\}$.\\
In this section we obtain the following results:
\begin{theoreme}\label{theorem2.1.1}
The sequence of random variables $\frac{\log{\sqrt{\tau_{\epsilon}}}}{-\log{\epsilon}}$ converges $\nu$-almost everywhere as $\epsilon \rightarrow 0$ to the Hausdorff dimension $d$ of the measure $\nu$.
\end{theoreme}
\begin{theoreme}\label{theorem2.2}
The sequence of random variables $\nu((B_{\epsilon}(.))\sqrt{\tau_{\epsilon}(.)}$ converges in distribution with respect to every probability measure absolutely continuous with respect to $\nu$ as $\epsilon \rightarrow 0$ to $\frac{\mathcal{E}}{\mid\mathcal{N}\mid}$, where $\mathcal{E}$ and $\mathcal{N}$ are independent random variables, $\mathcal{E}$ having an exponential distribution of mean 1 and $\mathcal{N}$ having a standard Gaussian distribution.
\end{theoreme}
\begin{corollary}\label{corollary}
If the measure $\nu$ is not the measure of maximal entropy, then the sequence of random variables $\frac{\log\sqrt{\tau_{\epsilon}}+d\log{\epsilon}}{\sqrt{-\log\epsilon}}$ converges in distribution as $\epsilon\rightarrow 0$ to a centered Gaussian random variable of variance $2\sigma^{2}_{h}$.\\
\end{corollary}
\subsection{Spectral theory of the transfer operator and Local Limit Theorem} In this subsection, we follow  \cite{penesaussol} to adapt our results.
To begin with, let us define:
\begin{equation*}
\hat{\Sigma}:=\{w:=(w_{n})_{n\in \mathbb{N}}: \forall n\in \mathbb{N}, M(w_{n},w_{n+1})=1\},
\end{equation*}
the set of all one-sided infinite sequences of elements of $\mathcal{A}$, endowed with the metric $\hat{d}((w_{n})_{n\geq0},(w_{n}^{'})_{n\geq0}):=e^{-\inf\{m\geq0 : w_{m}\neq w_{m}^{'}\}}$, and the one-sided shift map $\hat{\theta}((w_{n})_{n\geq0})=(w_{n+1})_{n\geq0}$. The resulting topology is generated by the collection of cylinders:
\begin{equation*}
C_{a_{0},...,a_{n}}=\{(w_{n})_{n\in \mathbb{N}}\in \hat{\Sigma}:w_{0}=a_{0},...,w_{n}=a_{n}\}.
\end{equation*}
Let us introduce the canonical projection $\Pi : \Sigma \rightarrow \hat{\Sigma}$, $\Pi((w_{n})_{n\in \mathbb{Z}})=(w_{n})_{n\geq0}$. Denote by $\hat{\nu}$ the image probability measure (on $ \hat{\Sigma})$ of $ \nu$ by $\Pi$.
There exists a function $\psi : \hat{\Sigma} \rightarrow \mathbb{Z}$ such that $\psi \circ \Pi = \varphi \circ \theta^{m_{0}}$.\\
let us denote by $P:L^{2}(\hat{\nu}) \rightarrow L^{2}(\hat{\nu})$ the Perron-Frobenius operator such that:
\begin{equation*}
\forall f,g \in L^{2}(\hat{\nu}), \int_{\hat{\Sigma}}Pf(x)g(x)d\hat{\nu}(x)=\int_{\hat{\Sigma}}f(x)g\circ \hat{\theta}(x)d\hat{\nu}(x).
\end{equation*}
Let $\eta \in ]0;1[$. Let us denote by $\mathcal{B}$ the set of bounded $\eta$-H\"{o}lder continuous function $g : \hat{\Sigma} \rightarrow \mathbb{C}$ endowed with the usual H\"{o}lder norm :
\begin{equation*}
||g||_{\mathcal{B}}:=||g||_{\infty} + \underset{x\neq y}{\sup}\frac{|g(y)-g(x)|}{\hat{d}(x,y)^{\eta}}.
\end{equation*}
We denote by $\mathcal{B}^{*}$ the topological dual of $\mathcal{B}$. For all $u\in \mathbb{R}$, we consider the operator $P_{u}$ defined on $(\mathcal{B},||.||_{\mathcal{B}})$ by:
\begin{equation*}
P_{u}(f):=P(e^{iu\psi}f).
\end{equation*}
 Note that the hypothesis of non-arithmeticity of $\varphi$ is equivalent to the following one on $\psi$:
for any $u \in [-\pi; \pi]\backslash\{0\}$, the operator $P_{u}$ has no eigenvalue on the unit circle.\\
 We will use the method introduced by Nagaev in \cite{nagaev1} and \cite{nagaev2}, adapted by Guivarch and Hardy in \cite{guivarch} and extended by Hennion and Herv\'{e} in \cite{hennion}. It is based on the family of operators $(P_{u})_{u}$ and their spectral properties expressed in the two next propositions.
\begin{proposition}\label{p1}(Uniform Contraction).
There exist $\alpha \in(0;1)$ and $ C>0$ such that, for all $u\in [-\pi;\pi]\backslash[-\beta;\beta]$ and all integer $n\geq0$, for all $f\in \mathcal{B}$, we have:
\begin{equation}
||P^{n}_{u}(f)||_{\mathcal{B}}\leq C\alpha^{n}||f||_{\mathcal{B}}.
\end{equation}
\end{proposition}
\hspace{1mm}This property easily follows from the fact that the spectral radius is smaller than 1 for $u\neq 0$.
In addition, since $P$ is a quasicompact operator on $\mathcal{B}$ and since $u\mapsto P_{u}$ is a regular perturbation of $P_{0}=P$, we have :
\begin{proposition}\label{p2}(Perturbation Result).
There exist $\alpha > 0, \beta > 0, C > 0,c_{1} > 0, \theta \in ]0;1[  $ such that: there exists $u\mapsto \lambda_{u}$ belonging to $C^{3}([-\beta;\beta] \rightarrow \mathbb{C})$, there exists $u\mapsto v_{u}$ belonging to $C^{3}([-\beta;\beta] \rightarrow \mathcal{B})$, there exists $u \mapsto \varphi_{u}$ belonging to $C^{3}([-\beta;\beta] \rightarrow \mathcal{B^{\ast}})$ such that, for all $u\in [-\beta; \beta]$, for all $f \in \mathcal{B}$ and for all $n \geq 0 $, we have the decomposition:
\begin{equation*}
P^{n}_{u}(f)=\lambda_{u}^{n}\varphi_{u}(f)v_{u} + N_{u}^{n}(f),
\end{equation*}
with
\begin{enumerate}
\item $||N_{u}^{n}(f)||_{\mathcal{B}}\leq {C \alpha^{n}||f||_{\mathcal{B}}}$,\label{p21}
\item $|\lambda_{u}|\leq e^{-c_{1}|u|^{2}}$ and $c_{1}|u|^{2}\leq \sigma^{2}_{\phi}u.u$,\label{p22}
\item with initial values : $v_{0}=\mathds{1} , \phi_{0}=\hat{\nu}, \lambda_{u=0}^{'}=0$ and $\lambda_{u=0}^{''}=-\sigma^{2}_{\varphi}$.
\end{enumerate}
\end{proposition}
\begin{lemma}\label{lemma LLT}
 There exist $\gamma^{'}>0$ and $C_{\eta}>0$ such that, $\forall q\geq m_{0}$ and all $2q$-cylinder $\hat{A}$ of $\hat{\Sigma}$, we have:
\begin{equation}
\forall u\in[-\pi,\pi],\quad ||P_{u}^{q}P^{q}(1_{\hat{A}}\circ\hat{\theta}^{m_{0}})||_{\mathcal{B}}\leq C_{\eta} e^{-\gamma^{'}(2q-m_{0})}.
\end{equation}
In particular, we have $\hat{\nu}(\hat{A})\leq C_{\eta}e^{-\gamma^{'}(2q-m_{0})}$.
\end{lemma}
\begin{proof}
\begin{eqnarray*}
P_{u}^{q}P^{q}(1_{\hat{A}}\circ\hat{\theta}^{m_{0}})(y)
&=&P^{q}\left(e^{iuS_{q}\psi}P^{q-m_{0}}(1_{\hat{A}})\right)(y)\\
&=&\sum_{w:\hat{\theta}^{2q-m_{0}}w=y}e^{S_{2q-m_{0}}h(w)}1_{\hat{A}}(w)e^{iuS_{q}\psi\hat{\theta}^{q-m_{0}}(w)}\\
&=&1_{[\hat{\theta}^{2q-m_{0}}\hat{A}]}(y)e^{S_{2q-m_{0}}h(w_{y})}e^{iuS_{q}\psi\hat{\theta}^{q-m_{0}}(w_{y})}.
\end{eqnarray*}
where $w_{y}\in \hat{A}$ is the unique element such that $\hat{\theta}^{2q-m_{0}}w_{y}=y$ (it exists if $y\in\hat{\theta}^{2q-m_{0}}\hat{A}$). From this later formula, we can obtain that $|||P_{u}^{q}P^{q}(1_{\hat{A}}\circ\hat{\theta}^{m_{0}})||_{\infty}\leq e^{\max h(2q-m_{0})}$, where $\max h<0$.\\
Now a step to compute the norm $||.||_{\mathcal{B}}$ is to estimate the $\eta-$H$\ddot{o}$lder coefficient. Let $x\neq y\in \hat{\Sigma}$, we know that $\hat{d}(x,y)=e^{-n}$, for some $n\in \mathbb{N}^{*}$. We will consider two cases: the first case when $n>m_{0}$, we note the equivalence $x\in \hat{\theta}^{2q-m_{0}}\hat{A} \Leftrightarrow y\in \hat{\theta}^{2q-m_{0}}\hat{A}$. Thus, either $x,y\notin \hat{\theta}^{2q-m_{0}}\hat{A}$ and hence
\begin{equation*}
|P_{u}^{q}P^{q-m_{0}}(1_{\hat{A}})(y)-P_{u}^{q}P^{q-m_{0}}(1_{\hat{A}})(x)|=0.
\end{equation*}
Or $x,y\in \hat{\theta}^{2q-m_{0}}\hat{A}$, so that $\hat{d}(w_{x},w_{y})=2q-m_{0}+n$. Let us denote for simplicity $F_{h,\psi}=S_{2q-m_{0}}h(.)+iuS_{q}\psi\circ\hat{\theta}^{q-m_{0}}(.)$. Introducing the ergodic sum formula, we get:
\begin{equation*}
|S_{2q-m_{0}}h(w_{y})-S_{2q-m_{0}}h(w_{x})|
\leq\sum_{i=0}^{2q-m_{0}-1}|h|_{\eta}e^{-\eta(2q-m_{0}+n-i)}\\
\leq c |h|_{\eta}\hat{d}^{\eta}(x,y),
\end{equation*}
where $c$ is a constant such that $\sum_{j\geq 1}e^{-\alpha j}\leq c<\infty$. And in the same way for $S_{q}\psi(\hat{\theta}^{q-m_{0}}(.))$, we can see that:
\begin{equation*}
|S_{q}\psi(\hat{\theta}^{q-m_{0}}(w_{y}))-S_{q}\psi(\hat{\theta}^{q-m_{0}}(w_{x}))|
\leq c |\psi|_{\eta}\hat{d}^{\eta}(x,y).
\end{equation*}
Thus, from these computations, we verify that:
\begin{eqnarray*}
|P_{u}^{q}P^{q-m_{0}}(1_{\hat{A}}(y))-P_{u}^{q}P^{q-m_{0}}(1_{\hat{A}}(x))|
&=&\mid e^{F_{h,\psi}(w_{y})}-e^{F_{h,\psi}(w_{x})}\mid\\
&\leq&e^{\max h (2q-m_{0})}c(|h|_{\eta}+|\psi|_{\eta})\hat{d}^{\eta}(x,y).
\end{eqnarray*}
Now, we treat the second case where $n\leq m_{0}$. Here, if $x\in\hat{\theta}^{2q-m_{0}}\hat{A}$, then $y\notin \hat{\theta}^{2q-m_{0}}\hat{A}$,
\begin{eqnarray*}
|P_{u}^{q}P^{q-m_{0}}(1_{\hat{A}})(y)-P_{u}^{q}P^{q-m_{0}}(1_{\hat{A}})(x)|
&\leq&\sup{|P_{u}^{q}P^{q-m_{0}}(1_{\hat{A}})|}\\
&\leq&\underset{w\in \hat{A}}{ \sup}|e^{S_{2q-m_{0}}h(w)+iu S_{q}\psi\circ\hat{\theta}^{q-m_{0}}(w)}|e^{\eta n}e^{-\eta n}\\
&\leq& e^{\max h(2q-m_{0})}e^{\eta m_{0}}{\hat{d}}^{\eta}(x,y).
\end{eqnarray*}
{}From all this process, setting $\gamma^{'}:=\min(\eta,-\max h)>0$, we get an estimation for the $\eta$-H\"{o}lder coefficient, $\forall n\geq0$:
\begin{equation*}
|P_{u}^{q}P^{q-m_{0}}(1_{\hat{A}})|_{\eta}\leq e^{-\gamma^{'}(2q-m_{0})}\max\left(e^{\eta m_{0}}, c(|h|_{\eta}+|\psi|_{\eta})\right)
\end{equation*}
Hence, for $C_{\eta}:=(1+\max\left(e^{\eta m_{0}}, c(|h|_{\eta}+|\psi|_{\eta})\right))$, we deduce that
\begin{equation*}
||P_{u}^{q}P^{q-m_{0}}(1_{\hat{A}})||_{\mathcal{B}}\leq C_{\eta}e^{-\gamma^{'}(2q-m_{0})}.
\end{equation*}
\end{proof}
 Next proposition is a two-dimensional version of Proposition 13 in \cite{penesaussol}. We give a more precise error term in order to accomodate the one-dimensional case. It may be viewed as a doubly local version of the central limit theorem: first, it is local in the sense that we are looking at the probability that $S_{n}\varphi=0$ while the classical central limit theorem is only concerned with the probability that $|S_{n}\varphi|\leq \epsilon \sqrt{n}$; second, it is local in the sense that we are looking at this probability conditioned to the fact that we are starting from a set $A$ and landing on a set $B$ on the base.
\begin{proposition}\label{p3}
There exist real numbers $C_{1}>0$ and $\gamma>0$ such that, for all integers $n$, $q$, $k$ such that $n-2k> 0$ and all $m_{0}<q\leq k$, all two-sided $q$-cylinders $A$ of $\Sigma$ and all measurable subset $B$ of $\hat{\Sigma}$, we have:
\begin{equation*}
\bigg|\nu\left(A\cap\{S_{n}\varphi=0\}\cap \theta^{-n}(\theta^{k}(\Pi^{-1}(B)))\right)-\frac{\nu(A)\hat{\nu}(B)}{ \sqrt{n-k}\sigma_{\varphi}}\bigg| \leq C_{1}\frac{\hat{\nu}(B)k^{2}e^{-\gamma q}}{n-2k}.
\end{equation*}
\end{proposition}
\begin{proof}
Set $Q:=A\cap\{S_{n}\varphi=0\}\cap \theta^{-n}(\theta^{k}(\Pi^{-1}(B))$. The proof of the proposition will be illustrated in  estimating the measure of the set $Q$.\\
Since $\varphi\circ \theta^{m_{0}}=\psi \circ \Pi$ and using the semi-conjugacy $\hat{\theta}\circ\Pi=\Pi\circ\theta$, we have the identity: $\{S_{n}\varphi \circ \theta^{m_{0}}=0\}=\{S_{n}\psi\circ \Pi=0\}$. In addition, $Im(\psi)\in\mathbb{Z}$, thus we have:
\begin{equation*}
1_{\theta^{-q-m_{0}}Q}=\left(1_{\hat{A}}\circ\hat{\theta}^{m_{0}}.1_{B}\circ\hat{\theta}^{q+n-(k-m_{0})}.\frac{1}{2\pi}\int e^{iu.S_{n}\psi\circ\hat{\theta}^{q}}du\right)\circ\Pi,
\end{equation*}
with $\hat{A}:=\Pi \theta^{-q}A$( indeed $\theta^{-q}A=\Pi^{-1}\hat{A}$ since $A$ is a $q$-cylinder).
Since the measure $\nu$ is $\theta$-invariant, then we can verify that:
\begin{equation*}
\nu(Q)=\frac{1}{2\pi}\int_{[-\pi,\pi]}\mathbb{E}_{\hat{\nu}}\left(1_{\hat{A}}\circ\hat{\theta}^{m_{0}}.1_{B}\circ\hat{\theta}^{q+n-(k-m_{0})} e^{iu.S_{n}\psi\circ\hat{\theta}^{q}}\right)du.
\end{equation*}
Now we want to estimate the expectation $a(u)=\mathbb{E}_{\hat{\nu}}(...)$. Introducing the Perron-Frobenius operator P, and using the fact that it is the dual of $\hat{\theta}$, we get:
\begin{eqnarray*}
a(u)&=&\mathbb{E}_{\hat{\nu}}\left( P^{q}(1_{\hat{A}}\circ\hat{\theta}^{m_{0}})\exp(iu.S_{n}\psi)1_{B}\circ\hat{\theta}^{n-(k-m_{0})}\right)\\
&=&\mathbb{E}_{\hat{\nu}}\left( P^{n}_{u}\left(P^{q}(1_{\hat{A}}\circ\hat{\theta}^{m_{0}})1_{B}\circ\hat{\theta}^{n-(k-m_{0})} \right)\right)\\
&=&\mathbb{E}_{\hat{\nu}}\left( P^{k-m_{0}}_{u}(1_{B}P_{u}^{n-(k-m_{0})}P^{q}(1_{\hat{A}}\circ\hat{\theta}^{m_{0}}))\right).
\end{eqnarray*}
We will treat two cases concerning the values of $u$. Let us denote for simplicity $l:=n-(k-m_{0}-q)$. First, using the contraction inequality given in Proposition \ref{p1} applied to $P^{l}_{u}(1)$, the fact that $||P_{u}^{q}P^{q}(1_{\hat{A}}\circ\theta^{m_{0}})||_{B}\leq e^{-\gamma^{'} (2q-m_{0})}$ from Lemma \ref{lemma LLT}, and the fact that $\bigg|\mathbb{E}\left(P^{k-m_{0}}_{u}(1_{B}g)\right)\bigg|\leq \hat{\nu}(B)||g||_{\mathcal{B}}$, we will show that $a(u)$ is negligible for large values of $u$, so when $u\notin[-\beta, \beta]$ we get for $\gamma=2\gamma^{'}$:
\begin{equation*}\label{3}
|a(u)|=\bigg|\mathbb{E}_{\hat{\nu}}\left( P^{k-m_{0}}_{u}(1_{B}P_{u}^{n-(k-m_{0})-q}P_{u}^{q}P^{q}(1_{\hat{A}}\circ\hat{\theta}^{m_{0}}))\right)\bigg|\\
=O\left(\hat{\nu}(B)\alpha^{l}e^{-\gamma q}\right). 
\end{equation*}
We now use the decomposition in \ref{p2} to obtain an estimation of the main term coming from small values of u. Indeed, whenever $u\in[-\beta,\beta]$, we have:
\begin{eqnarray*}
a(u)
&=&\mathbb{E}_{\hat{\nu}}\left( P^{k-m_{0}}_{u}(1_{B}P_{u}^{l}P_{u}^{q}P^{q}(1_{\hat{A}}\circ\hat{\theta}^{m_{0}}))\right)\\
&=&\lambda_{u}^{l}\varphi_{u}(P_{u}^{q}P^{q}(1_{\hat{A}}\circ\hat{\theta}^{m_{0}}))\mathbb{E}_{\hat{\nu}}\left(P_{u}^{k-m_{0}}(1_{B}v_{u})\right)+\mathbb{E}_{\hat{\nu}}\left(P^{k-m_{0}}_{u}(1_{B}N_{u}^{l}(P_{u}^{q}P^{q}(1_{\hat{A}}\circ\hat{\theta}^{m_{0}}))\right)\\
&=&a_{1}(u) + a_{2}(u).
\end{eqnarray*}
Using inequality $(1)$ in Proposition \ref{p2}, one can see that the second term is of order
\begin{equation}\label{4}
a_{2}(u)=O(\hat{\nu}(B)\alpha^{l}e^{-\gamma q}).
\end{equation}
The mappings $u \mapsto v_{u}$ and $u\mapsto \phi_{u}$ are $C^{1}$-regular with $v_{0}=1$ and $\varphi_{0}=\hat{\nu}$, from which we find that:
\begin{eqnarray*}
a_{1}(u)
&=&\lambda^{l}_{u}\hat{\nu}(P_{u}^{q}P^{q}(1_{\hat{A}}\circ\hat{\theta}^{m_{0}}))\mathbb{E}_{\hat{\nu}}\left(P_{u}^{k-m_{0}}(1_{B})\right) +\lambda^{l}_{u}\hat{\nu}(P_{u}^{q}P^{q}(1_{\hat{A}}\circ\hat{\theta}^{m_{0}}))\mathbb{E}_{\hat{\nu}}\left(P_{u}^{k-m_{0}}(1_{B}O(u))\right)\\ & &+\lambda^{l}_{u}O(u)(P_{u}^{q}P^{q}(1_{\hat{A}}\circ\hat{\theta}^{m_{0}}))\mathbb{E}_{\hat{\nu}}\left(P_{u}^{k-m_{0}}(1_{B})\right) +\lambda^{l}_{u}O(u)(P_{u}^{q}P^{q}(1_{\hat{A}}\circ\hat{\theta}^{m_{0}}))\mathbb{E}_{\hat{\nu}}\left(P_{u}^{k-m_{0}}(1_{B})O(u))\right)
\end{eqnarray*}
To obtain an approximation of the first term $a_{1}(u)$, we introduce the formula of $P$ in $P_{u}$:
\begin{eqnarray*}
\bigg|\mathbb{E}_{\hat{\nu}}\left(P_{u}^{k-m_{0}}(1_{B})\right)-\hat{\nu}(B)\bigg|
&=&|\mathbb{E}_{\hat{\nu}}\left(P^{k-m_{0}}(e^{iu.S_{k-m_{0}}\psi}-1)1_{B} \right)|\\
&\leq&||e^{iu.S_{k-m_{0}}\psi}-1||_{\infty}||1_{B}||_{L^{1}(\hat{\nu})}\\
&\leq& |u|.(k-m_{0})||\psi||_{\infty}\hat{\nu}(B),
\end{eqnarray*}
so that, from this approximation, we get:
\begin{eqnarray*}
a_{1}(u)
&=&\lambda_{u}^{l}\hat{\nu}(P^{q}_{u}P^{q}(1_{\hat{A}}\circ\hat{\theta}^{m_{0}}))\mathbb{E}_{\hat{\nu}}\left(P_{u}^{k-m_{0}}(1_{B})\right)+O(\lambda^{l}_{u}|u|\hat{\nu}(B)e^{-\gamma q})\\
&=&\lambda_{u}^{l}\hat{\nu}(\hat{A})\hat{\nu}(B)\left(1+O(|u|q)\right)\left(1+O(\mid u\mid(k-m_{0})\right)+O(\lambda^{l}_{u}|u|\hat{\nu}(B)e^{-\gamma q})\\
&=&\lambda_{u}^{l}\hat{\nu}(\hat{A})\hat{\nu}(B)+O(\lambda^{l}_{u}|u|\hat{\nu}(B)k^{2}e^{-\gamma q}).
\end{eqnarray*}
Using Proposition \ref{p2} and that $u\mapsto \lambda_{u}$ belongs to $C^{3}([-\beta;\beta] \rightarrow \mathbb{C})$, hence applying the intermediate value theorem
 gives:
\begin{eqnarray*}
|\lambda_{u}^{l}-e^{-\frac{l}{2}\sigma^{2}_{\varphi}u^{2}}|
&\leq&l(e^{-c_{1}|u|^{2}})^{l-1}|\lambda_{u}-e^{-\frac{1}{2}\sigma^{2}_{\varphi}u^{2}}|\\
&=&le^{-c_{1}l|u|^{2}}e^{c_{1}|u|^{2}}O(|u|^{3})\\
&=&C_{0}\left(l|u|^{2}e^{-c_{1}l|u|^{2}}\right)e^{c_{1}|u|^{2}}|u|\\
&=&O(e^{-c_{2}l|u|^{2}}|u|),\quad \text{ for the constant }  c_{2}=c_{1}/2.
\end{eqnarray*}
As a consequence, an estimate for $a_{1}(u)$ is:
\begin{equation*}
a_{1}(u)=e^{-\frac{l}{2}\sigma^{2}_{\varphi}u^{2}}\hat{\nu}(\hat{A})\hat{\nu}(B)+ O(e^{-c_{2}l|u|^{2}}|u|\hat{\nu}(B)ke^{-\gamma q}),
\end{equation*}
since $\hat{\nu}(\hat{A})=O(e^{-\gamma^{'}(2q-m_{0})})$.
A final step to reach an estimation of $\nu(Q)$ is to integrate the approximated quantity of $a_{1}(u)$ obtained above. Using the Gaussian integral, a change of variable $v=u\sqrt{l}$ gives:
\begin{equation*}
\int_{[-\beta,\beta]}e^{-\frac{l}{2}\sigma^{2}_{\varphi}u^{2}}du
=\frac{1}{\sqrt{l}}\frac{2\pi}{\sigma_{\varphi}} + O\left(\frac{1}{l}\right)
\end{equation*}
In the same way we treat the error term to get:
\begin{equation*}
\int_{[-\beta,\beta]}|u|e^{-c_{2}l|u|^{2}}du
=\frac{1}{l}\int_{[-\beta\sqrt{l},\beta\sqrt{l}]}|v|e^{-c_{2}|v|^{2}}dv\\
=O\left(\frac{1}{l}\right).
\end{equation*}
From these computations, it follows that:
\begin{equation*}
\int_{[-\beta,\beta]}a_{1}(u)du
=\frac{2\pi}{\sqrt{l\sigma^{2}_{\varphi}}}\hat{\nu}(A)\hat{\nu}(B) + O\left(\frac{\hat{\nu}(B)k^{2}e^{-\gamma q}}{l}\right).
\end{equation*}
From this main estimate and (\ref{3}) and (\ref{4}) we conclude that:
\begin{equation*}
\nu(Q)
=\frac{1}{2\pi}\int_{[-\pi,\pi]}a(u)du\\
=\frac{1}{{\sqrt{n-k}\sigma_{\varphi}}}\hat{\nu}(A)\hat{\nu}(B) + O\left(\frac{\hat{\nu}(B)k^2e^{-\gamma q}}{n-2k}\right)
\end{equation*}
\end{proof}
\subsection{Proof of the pointwise convergence of the recurrence rate to the dimension. }
\hspace{2mm} Let us denote by $G_{n}(\epsilon)$ the set of points for which $n$ is an $\epsilon$-return :
\begin{equation*}
G_{n}(\epsilon):= \{x\in \Sigma : S_{n}\varphi(x)=0 \text{ and }d(\theta^{n}(x),x)<\epsilon\}.
\end{equation*}
Let us consider the first return time in an $\epsilon$-neighborhood of a starting point $x\in \Sigma$ :
\begin{equation*}
\tau_{\epsilon}(x):= \inf\{m\geq 1 : S_{m}\varphi(x)=0 \text{ and }d(\theta^{m}(x),x)<\epsilon\}=\inf\{m\geq1 : x\in G_{m}(\epsilon)\}.
\end{equation*}
\begin{proof}[Proof of Theorem \ref{theorem2.1.1}]Let us denote by $\mathcal{C}_{k}$ the set of two-sided $k$-cylinders of $\Sigma$. For any $\delta>0$ denote by $\mathcal{C}_{k}^{\delta}\subset\mathcal{C}_{k}$ the set of cylinders $C \in \mathcal{C}_{k}$ such that $\nu(C)\in (e^{-(d+\delta)k},e^{-(d-\delta)k})$. For any $x\in \Sigma$, let $C_{k}(x)\in \mathcal{C}_{k}$ be the $k$-cylinder which contains $x$. Since $d$ is twice the entropy of the ergodic measure $\nu$, by the Shannon-Breiman theorem, the set $K_{N}^{\delta}=\{x\in \Sigma: \forall k \geq N , C_{k}(x)\in  \mathcal{C}^{\delta}_{k}\}$ has a measure $\nu(K_{N}^{\delta})> 1-\delta$ provided $N$ is sufficiently large.
\begin{itemize}
\item First, let us prove that, almost surely :
\begin{equation*}
\underset{\epsilon \rightarrow 0}{\liminf}\frac{\log\sqrt{\tau_{\epsilon}}}{-\log\epsilon}\geq d.
\end{equation*}
Let $\alpha > \frac{1}{d}$ and $0<\delta<d-\frac{1}{\alpha}$. Let us take $\epsilon_{n}:=n^{-\frac{\alpha}{2}}$ and $k_{n}:=\lceil-\log \epsilon_{n} \rceil$. In view of Proposition~\ref{p3}, whenever $k_{n}\geq N$, we have :
\begin{eqnarray*}
\nu(K_{N}^{\delta}\cap G_{n}(\epsilon_{n}))
&=&\nu\left(\{x\in K_{N}^{\delta}:  S_{n}\varphi(x)=0 \text{ and } \theta^{n}(x)\in C_{k_{n}}(x)\}\right)\\
&=&\sum_{C\in \mathcal{C}^{\delta}_{k_{n}}}\nu(C\cap \{S_{n}\varphi=0\}\cap\theta^{-n}(C))\\
&=&\sum_{C\in \mathcal{C}^{\delta}_{k_{n}}}\frac{\nu(C)^{2}}{\sigma_{\varphi}\sqrt{n}} + O\left(\frac{\nu(C)k_{n}^{2}e^{-\gamma k_{n}}}{n-2k_{n}} \right).
\end{eqnarray*}
Notice that for $\epsilon_{n}$ and $k_{n}$ taken as above, one can verify that the term $\frac{k_{n}^{2}e^{-\gamma k_{n}}}{n-2k_{n}}=O(n^{-1-\frac{\gamma\alpha}{2}}(\log n)^{2})$. In addition, for $C\in C_{k_{n}^{\delta}}$, $\nu(C)\leq n^{-\frac{\alpha(d-\delta)}{2}}$, from which it follows that
\begin{equation*}
\nu(K_{N}^{\delta}\cap G_{n}(\epsilon_{n}))
=O\left(\frac{(\log n)^{2}}{n^{\min\left(1+\frac{\gamma\alpha}{2},1+\frac{\alpha(d-\delta)}{2}\right)}}\right)
\end{equation*}
but $\frac{1+\alpha(d-\delta)}{2}>1$, so $\sum_{n}\nu(K_{N}^{\delta}\cap G_{n}(\epsilon_{n}))< \infty$.\\
Hence by the Borel Cantelli lemma, for a.e. $x\in K_{N}^{\delta}$, if n is large enough, we have  $\tau_{\epsilon_{n}}> n$, which in turn implies that :
\begin{equation*}
\underset{n\rightarrow \infty}{\liminf}\frac{\log\sqrt{\tau_{\epsilon_{n}}}}{-\log\epsilon_{n}}\geq \frac{1}{\alpha} \quad a.e.,
\end{equation*}
and this proves the the lower bound on the $\liminf$, since $(\epsilon_{n})_{n}$ decreases to zero and $\underset{n\rightarrow+\infty}{\liminf}\frac{\epsilon_{n}}{\epsilon_{n+1}}=1$, and since we have taken an arbitrary $\alpha>\frac{1}{d}$.
\item Next, we will prove the upper bound $(d)$ on the lim sup :
\begin{equation*}
\underset{\epsilon\rightarrow 0}{\limsup}\frac{\log\sqrt{\tau_{\epsilon}}}{-\log\epsilon}\leq d.
\end{equation*}
let $\alpha \in(0,\frac{1}{d})$ and $\delta>0$ such that $1-\alpha d-\alpha\delta>0$. Take $\epsilon_{n}:=n^{-\frac{\alpha}{2}}$ and $k_{n}:=\lceil-\log \epsilon_{n} \rceil$. We define for all $l=1,...,n$, the sets $ A_{l}(\epsilon):=G_{l}(\epsilon)\cap \theta^{-l}\{\tau_{\epsilon}>n-l\}$ which are pairwise disjoint. Setting $L_{n}:=\lceil n^{a}\rceil$ with $a>\alpha(d+\delta -\gamma)$, we then realize that:
\begin{equation} \label{eqnAl}
1\geq \sum_{l=0}^{n}\nu(A_{l}(\epsilon_{n}))\geq \sum_{l=L_{n}}^{n}\sum_{C\in\mathcal{C}^{\delta}_{k_{n}}}\nu(C\cap A_{l}(\epsilon_{n})).
\end{equation}
But due to Proposition \ref{p3}, for any $C\in\mathcal{C}^{\delta}_{k_{n}}$ and $l\geq L_{n}$, whenever $k_{n}\geq N$, we have :
\begin{eqnarray*}
\nu(C\cap A_{l}(\epsilon_{n}))
&=&\nu(C\cap\{S_{l}\varphi=0\}\cap\theta^{-l}(C\cap\{\tau_{\epsilon_{n}}>n-l\}))\\
&=&\frac{\nu(C)\nu(C\cap\{\tau_{\epsilon_{n}}>n-l\})}{\sigma_{\varphi}\sqrt{l-k_{n}}} + O\left(\frac{\nu(C\cap\{\tau_{\epsilon_{n}}>n-l\})k^{2}_{n}e^{-\gamma k_{n}}}{l}\right)\\
&\geq& c \epsilon_{n}^{d+\delta}\frac{1}{\sqrt{l}}\nu(C\cap\{\tau_{\epsilon_{n}}>n-l\}).
\end{eqnarray*}
Indeed, the error is negligible, because for $a>\alpha(d+\delta -\gamma)$, $\frac{k^{2}_{n}e^{-\gamma k_{n}}}{\sqrt{l}}=O(\epsilon^{d+\delta}_{n})$.\\
Now, note that:
\begin{equation*}
\nu\left(K_{N}^{\delta}\cap\{\tau_{\epsilon_{n}}>n\}\right)\leq \sum_{C\in \mathcal{C}^{\delta}_{k_{n}}}\nu\left(C\cap\{\tau_{\epsilon_{n}}>n\}\right).
\end{equation*}
Next, we will work to prove that $\nu\left(K_{N}^{\delta}\cap\{\tau_{\epsilon_{n}}>n\}\right)$ is summable.\\
Observe that:
\begin{equation*}
\sum_{l=L_{n}}^{n}\nu(C\cap A_{l}(\epsilon_{n}))
\geq c\epsilon^{d+\delta}_{n}\nu(C\cap\{\tau_{\epsilon_{n}}>n\})\left(\sqrt{n}-\sqrt{L_{n}}\right).
\end{equation*}
But, from  \eqref{eqnAl}, it follows immediately that
\begin{equation*}
1\geq \sum_{C\in \mathcal{C}^{\delta}_{k_{n}}}\sum_{l=L_{n}}^{n}\nu(C\cap A_{l}(\epsilon_{n}))\geq \sum_{C\in \mathcal{C}^{\delta}_{k_{n}}}c\epsilon^{d+\delta}_{n}\nu(C\cap\{\tau_{\epsilon_{n}}>n\})\left(\sqrt{n}-\sqrt{L_{n}}\right)
\end{equation*}
from which one gets
\begin{equation*}
\nu\left(K_{N}^{\delta}\cap\{\tau_{\epsilon_{n}}>n\}\right)\leq\sum_{C\in \mathcal{C}^{\delta}_{k_{n}}}\nu(C\cap\{\tau_{\epsilon_{n}}>n\})=O\left(\frac{1}{n^{\frac{1-\alpha(d+\delta)}{2}}}\right).
\end{equation*}
Now let us take $n_{p}:=p^{-\frac{4}{1-\alpha d -\alpha\delta}}$. We have:
$\sum_{p\geq1}\nu(K_{N}^{\delta}\cap\{\tau_{\epsilon_{n_{p}}}>n_{p}\})$ is finite, revealing that, using Borel Cantelli lemma, almost surely $x\in K_{N}^{\delta}$,$\tau_{\epsilon_{n_{p}}}(x)\leq n_{p}$, which implies that :
\begin{equation*}
\underset{n\rightarrow +\infty}{\limsup}\frac{\log\sqrt{\tau_{\epsilon_{n_{p}}}}}{-\log\epsilon_{n_{p}}}\leq\frac{1}{\alpha}.
\end{equation*}
This gives the estimate $\limsup$ since $(\epsilon_{n_{p}})_{p}$ decreases to 0 and since $\underset{p\rightarrow +\infty}{\lim}\frac{\epsilon_{n_{p}}}{\epsilon_{n_{p+1}}}=1$.
\end{itemize}
\end{proof}
\subsection{Fluctuations of the rescaled return time. } Throughout this subsection, we adapt the general strategy of \cite{pzs,pzs2}.
Recall that $C_{k}(x)=\{y\in \Sigma : d(x,y)< e^{-k}\}.$ Let $R_{k}(y):=\min\{n\geq1 : \theta^{n}(y)\in C_{k}(y)\}$ denote the first return time of a point $y$ into its $k$-cylinder $C_{k}(y)$, or equivalently the first repetition time of the first $k$ symbols of $y$. We recall that $C_{k}(x)=\{y\in \Sigma : d(x,y)< e^{-k}\}.$ There have been a lot of studies on the quantity $R_{k}$, among all the results we will use the following.
\begin{proposition}(Hirata \cite{hirata})\label{prop hirata}
For $\nu$-almost every point $x\in \Sigma$, the return time into the cylinders $C_{k}(x)$ are asymptotically exponentially distributed in the sense that
\begin{equation*}
\underset{k\rightarrow +\infty}{\lim}\nu_{C_{k}(x)}\left(R_{k}(.)>\frac{t}{\nu(C_{k}(x))}\right)=e^{-t}
\end{equation*}
for a.e. $x$, where the convergence is uniform in $t$.
\end{proposition}
\begin{lemma}\label{lemma 3.4}
\begin{equation*}
\forall t>0, \quad 
\underset{k\rightarrow +\infty}{\limsup}\ \nu\left(\tau_{e^{-k}}>\left(\frac{t}{\nu(C_{k}(x))}\right)^{2}\bigg\vert C_{k}(x)\right)\leq\frac{1}{1+\beta t},
\end{equation*}
with $\beta:=\frac{1}{\sigma}.$
\end{lemma}
\begin{proof}
Let $k\geq m_{0}$ and $n$ be some integers. We make a partition of a cylinder $C_{k}(x)$ according to the value $l\leq n$ of the last passage in the time interval $0,...,n$ of the orbit of $(x,0)$ by the map $F$ into $C_{k}(x)\times \{0\}$. This gives the following equality :
\begin{equation}\label{5}
\nu(C_{k}(x))=\sum_{l=0}^{n}\nu\left(C_{k}(x)\cap \{S_{l}=0\}\cap \theta^{-l}(C_{k}(x)\cap \{\tau_{e^{-k}}>n-l\})\right).
\end{equation}
Let $n_{k}:=\left(\frac{t}{\nu(C_{k}(x))}\right)^{2}$. We claim that :
\begin{equation*}
\underset{k \rightarrow \infty}{\limsup}\nu(\{\tau_{e^{-k}}>n_{k}\}\mid C_{k}(x))\leq \frac{1}{1+\beta t}
\end{equation*}
According to the decomposition \eqref{5} and to Proposition \ref{p3}, there exists $c_{1}>0$ such that we have :
\begin{equation*}
\nu(C_{k}(x))\geq \nu(C_{k}(x)\cap \{\tau_{e^{-k}}> n_{k}\})\left(1+ \beta \nu(C_{k}(x))\sum_{l=2^{k+1}}^{n_{k}}\frac{1}{\sqrt{l-k}}\right)-c_{1}\nu(C_{k}(x))k^{2}e^{-\gamma k}\sum_{l=2^{k+1}}^{n_{k}}\frac{1}{l-2^{k}}
\end{equation*}
Our claim follows from the fact that
$\beta\nu(C_{k}(x))\sum_{l=2^{k+1}}^{l=n_{k}}\frac{1}{\sqrt{l-k}}\simeq \beta t$ and the term
$k^{2}e^{-\gamma k}\sum_{l=2^{k+1}}^{l=n_{k}}\frac{1}{l-2^{k}}\ll 1$.
\end{proof}
\begin{corollary}
The family of conditional distributions of the random variables $\left(\nu(C_{k}(x))\sqrt{\tau_{e^{-k}}}\mid C_{k}(x)\right)_{k\geq0}$ is tight.
\end{corollary}
Hence it will be enough to prove that the advertised limit law is the only possible accumulation point of our destination. We hence abbreviate
\begin{equation*}
X_{k}:=\nu(C_{k}(x))\sqrt{\tau_{e^{-k}}}
\end{equation*}
\begin{lemma}\label{lemma3.4.2}
Suppose that the sequence of conditional distributions of $(X_{k_{p}}\mid C_{k_{p}}(x))_{p}$ converges to the law of some random variable $X$. Then the limit satisfies the integral equation:
\begin{equation*}
1=\mathbb{P}(X>t)+\beta t\int_{0}^{1}\frac{\mathbb{P}(X>t\sqrt{1-u})}{\sqrt{u}}du \quad \forall t>0.
\end{equation*}
\end{lemma}
\begin{proof}
To begin with, let us set $f(t):=\mathbb{P}(X>t)$
\begin{itemize}
\item First, we will prove that
\begin{equation*}
\forall t>0 \quad 1\geq f(t)+\beta t\int_{0}^{1}u^{-1/2}f(t(1-u)^{1/2})du.
\end{equation*}
The decomposition in \eqref{5} and Proposition \ref{p3}, implies that there exists $c>0$ such that we have:
\begin{equation*}
1\geq\nu(\tau_{e^{-k}}>n_{k}\mid C_{k}(x))+\beta\nu(C_{k}(x))\sum_{l=1}^{n_{k}}\frac{\nu(\tau_{e^{-k}}>n_{k}-l\mid C_{k}(x))}{\sqrt{l-k}}-c\sum_{l=1}^{n_{k}}\frac{k^{2} e^{-\gamma k}\nu(C_{k}(x))}{l-2k}
\end{equation*}
We want to estimate the formula of this inequality when $k\rightarrow\infty$. We note that through the proof of Lemma \ref{lemma 3.4}, it has been proved that $\underset{k_{p}\rightarrow\infty}{\lim}\sum_{l=1}^{n_{k_{p}}}\frac{k_{p}e^{-\gamma k_{p}}}{l-k_{p}}=0$. Thus, if we set $B_{n_{k}}:=\sum_{l=1}^{n_{k}}\frac{\nu(\tau_{e^{-k}}>n_{k}-l\mid C_{k}(x))}{\sqrt{l-k}}$, we are left to estimate mainly the lower bound on the lim inf of  $\nu(C_{k_{p}}(x))B_{n_{k_{p}}}$ as $p\rightarrow\infty$.
Now, by monotonicity, we have
\begin{eqnarray*}
B_{n_{k}}&\geq&\sum_{l=\lfloor\frac{n_{k}}{N}\rfloor}^{N\lfloor\frac{n_{k}}{N}\rfloor}\frac{\nu(\tau_{e^{-k}}>n_{k}-l\mid C_{k}(x))}{\sqrt{l-k}}\\
&=&\sum_{r=1}^{N-1}\sum_{l=0}^{\lfloor\frac{n_{k}}{N}\rfloor}\frac{\nu(\tau_{e^{-k}}>n_{k}-l-(r\lfloor n_{k/N}\rfloor)\mid C_{k}(x))}{\sqrt{l+r\lfloor n_{k}/N}\rfloor}.
\end{eqnarray*}
Observe that the term:
\begin{eqnarray*}
\nu(\tau_{e^{-k}}>n_{k}-l-(r\lfloor n_{k/N}\rfloor)\mid C_{k}(x))
&\geq&\nu(\tau_{e^{-k}}>(1-r/N)n_{k}\mid C_{k}(x))\\
&=&\mathbb{P}\left(X_{k}>(\sqrt{1-r/N}t)\mid C_{k}(x)\right).
\end{eqnarray*}
Thus, now by evaluating the following sum:
\begin{equation*}
\sum_{l=0}^{\lfloor\frac{n_{k}}{N}\rfloor}\frac{1}{\sqrt{l+r\lfloor n_{k}/N}\rfloor}\geq\sqrt{\left\lfloor\frac{n_{k}}{ N}\right\rfloor}\frac{1}{\sqrt{r+1}},
\end{equation*}
We obtain
\begin{equation*}
B_{n_{k}}=\sum_{r=1}^{N-1}\sqrt{\left\lfloor\frac{n_{k}}{ N}\right\rfloor}\frac{1}{\sqrt{r+1}}\mathbb{P}\left(X_{k}>t\sqrt{(1-r/N)}\mid C_{k}(x)\right)
\end{equation*}
But, from the hypothesis that $\mathbb{P}\left(X_{k_{p}}>t\sqrt{1-r/N }\right)\overset{k_{p}\rightarrow \infty}{\longrightarrow}f\left(t\sqrt{1-r/N}\right)$, we get:
\begin{eqnarray*}
\underset{p\rightarrow\infty}{\liminf}\nu(C_{k_{p}}(x))B_{n_{k_{p}}}
&\geq&\frac{t}{N}\sum_{r=1}^{N-1}\frac{\mathbb{P}\left(X>t\sqrt{1-r/N}\right)}{\sqrt{(r+1)/N}}\\
&\geq&t\int_{0}^{1}\frac{f\left(t\sqrt{1-u}\right)}{\sqrt{u}}du.
\end{eqnarray*}
Combining these estimates and taking the limit when $k_{p}\rightarrow \infty$, we establish the desired inequality:
\begin{equation*}
1\geq f(t)+\beta t\int_{0}^{1}\frac{f\left(t\sqrt{1-u}\right)}{\sqrt{u}}du.
\end{equation*}
\item In the same way we treat the converse inequality, using the other half of Proposition \ref{p3}, then there exists $c^{'}>0$, such that:
\begin{equation*}
1\leq \nu(\tau_{e^{-k}}>n_{k}\mid C_{k}(x))+\beta\nu(C_{k}(x))\sum_{l=1}^{n_{k}}\frac{\nu(\tau_{e^{-k}}>n_{k}-l\mid C_{k}(x))}{\sqrt{l-k}}+c^{'}\sum_{l=1}^{n_{k}}\frac{k^{2} e^{-\gamma k}}{l-2k}
\end{equation*}
Let $m_{k}=o\left(\frac{1}{\nu(C_{k}(x))}\right)$, then using Proposition \ref{prop hirata}, we have
\begin{eqnarray*}
\nu(\tau_{e^{-k}}&\leq&1-\nu\left(\tau_{e^{-k}}>\frac{m_{k}\nu(C_{k})(x)}{\nu(C_{k}(x))}\right)\overset{k\rightarrow\infty}{\longrightarrow}1-e^{0}=0.
\end{eqnarray*}
from which we observe that we can forget the first $m_{k}$ term of the following sum, because
\begin{eqnarray*}
\sum_{l=1}^{m_{k}}&&\nu(C_{k}(x)\cap\{S_{l}=0\}\cap\theta^{-l}(C_{k}(x)\cap\{\tau_{e^{-k}}>m_{k}-l\}))\\
&=&\nu\left(\bigcup_{l=1}^{m_{k}}C_{k}(x)\cap\{S_{l}=0\}\cap\theta^{-l}(C_{k}(x)\cap\{\tau_{e^{-k}}>m_{k}-l\})\right)\\
&\leq&\nu(\{\tau_{e^{-k}}\leq m_{k}\}\cap C_{k}(x))\\
&=&o(\nu(C_{k}(x))).
\end{eqnarray*}
Furthermore, one verifies that this sum of terms between $m_{k}$ and $\lfloor n_{k}/N\rfloor$ is bounded above by $\frac{2t}{\nu(C_{k}(x))\sqrt{N}}$. Hence, we get:
\begin{eqnarray*}
1&\leq& \nu(\tau_{e^{-k}}>n_{k}\mid C_{k}(x))+\beta\nu(C_{k}(x))\sum_{l=\lfloor \frac{n_{k}}{N}\rfloor}^{n_{k}}\frac{\nu(\tau_{e^{-k}}>n_{k}-l\mid C_{k}(x))}{\sqrt{l-k}}\\
&+& c^{'}\sum_{l=1}^{n_{k}}\frac{k^{2} e^{-\gamma k}}{l-2k} + o(\nu(C_{k}(x)))+\beta\frac{2t}{\sqrt{N}},
\end{eqnarray*} where $N$ is so large that the last three terms goes to 0 as $k\rightarrow0$.
Moreover, if we set\\ $B_{n_{n_{k}}}^{'}:=\sum_{l=\lfloor n_{k}/N \rfloor}^{N\lfloor n_{k}/N\rfloor}\frac{\nu(\tau_{e^{-k}}>n_{k}-l\mid C_{k}(x))}{\sqrt{l}}$, we verify that
\begin{equation*}
\sum_{l=\lfloor \frac{n_{k}}{N}\rfloor}^{n_{k}}\frac{\nu\left(\tau_{e^{-k}}>n_{k}-l\mid C_{k}(x)\right)}{\sqrt{l-k}}\leq (1+\epsilon_{k})\left(B_{n_{k}}^{'}+N\frac{1}{\sqrt{n_{k}-N}}\right).
\end{equation*}
We now proceed to show the bound on the lim sup of $\nu(C_{k}(x))B_{n_{k_{p}}}^{'}$ as $p\rightarrow0$
\begin{eqnarray*}
B^{'}_{n_{k}}&=&\sum_{r=1}^{N-1}\sum_{l=r\lfloor\frac{n_{k}}{N}\rfloor}^{(r+1)\lfloor\frac{ n_{k}}{N}\rfloor-1}\frac{\nu\left(\tau_{e^{-k}}>n_{k}-l\mid C_{k}(x)\right)}{\sqrt{l}}\\
&\leq&\sum_{r=1}^{N-1}\sum_{l=0}^{\lfloor\frac{n_{k}}{N}\rfloor-1}\frac{\nu\left(\tau_{e^{-k}}>n_{k}-l-((r+1)\lfloor n_{k}/N\rfloor)\mid C_{k}(x)\right)}{\sqrt{l+r\lfloor n_{k}/N}\rfloor}.
\end{eqnarray*}
It can be easily seen that
\begin{equation*}
\sum_{l=0}^{\lfloor\frac{n_{k}}{N}\rfloor-1}\frac{1}{\sqrt{l+r\lfloor n_{k}/N}\rfloor}\leq \sqrt{\left\lfloor\frac{n_{k}}{ N}\right\rfloor}\frac{1}{\sqrt{r}},
\end{equation*}
hence, it follows immediately that
\begin{equation*}
B^{'}_{n_{k}}\leq\sum_{r=1}^{N-1}\sqrt{\left\lfloor\frac{n_{k}}{ N}\right\rfloor}\frac{1}{\sqrt{r}}\nu\left(\tau_{e^{-k}}>(1-(r+1)/N)n_{k}\mid C_{k}(x)\right).
\end{equation*}
Applying $\limsup$ when $p\rightarrow\infty$, then
\begin{equation*}
\underset{p\rightarrow\infty}{\limsup}\nu(C_{k_{p}}(x))\sum_{l=\lfloor n_{k_{p}}/N \rfloor}^{N\lfloor n_{k_{p}}/N\rfloor}\frac{\nu\left(\tau_{e^{-k_{p}}}>n_{k_{p}}-l\mid C_{k_{p}}(x)\right)}{\sqrt{l}}\leq t\int_{0}^{1}\frac{f(t\sqrt{1-u})}{\sqrt{u}}du.
\end{equation*}
Taking the limit when $k_{p}\rightarrow\infty$, and combining all these estimates, we get the second inequality:
\begin{equation*}
1\leq f(t)+\beta t\int_{0}^{1}\frac{f(t\sqrt{1-u})}{\sqrt{u}}du.
\end{equation*}
\end{itemize}
\end{proof}
\begin{lemma}\label{lemma3.4.3}
We know that the conditional distributions of the $X_{k_{p}}$ converge to a random variable $X$ iff the conditional distributions of the $X_{k_{p}}^{2}$ converge to $X^{2}$. The later then satisfies
\begin{equation*}
1=\mathbb{P}(X^2>t)+\int_{0}^{t}\frac{\mathbb{P}(X^{2}>t-v)}{\sqrt{v}}dv \qquad\forall t>0.
\end{equation*}
\end{lemma}
\begin{lemma}\label{lemma3.4.4}
 Let $W$ be a random walk variable with values in $[0,\infty[$ satisfying
 \begin{equation*}
 \mathbb{P}(W\leq t)=\int_{0}^{\infty}\frac{\mathbb{P}(W>t-v)}{\sqrt{v}}dv \qquad\forall t>0,
 \end{equation*}
 then
 \begin{equation*}
 \mathbb{E}\left[e^{-sW}\right]=\frac{1}{1+c_{\beta} \sqrt{s}}\qquad\forall s>0.
 \end{equation*}
 with $c_{\beta}:=\left(\beta\Gamma\left(\frac{1}{2}\right)\right)^{-1}$. In particular, the distribution of $W$ coincides with that of $c_{\beta}^{2}\dfrac{\mathcal{E}^{2}}{N^{2}}$, where the independent variables $\mathcal{E}$ and $\mathcal{N}$ are the exponential distribution of mean 1 and the standard Gaussian distribution respectively.
\end{lemma}
\begin{proof}
Let $s>0$. We have
\begin{equation*}
\mathbb{E}[e^{-sW}]=\int_{0}^{\infty}\mathbb{P}(e^{-sW}\geq u)du
=\int_{0}^{\infty}\mathbb{P}(W\leq v)se^{-sv}dv.
\end{equation*}
Hence, for any $s>0$, we find
\begin{eqnarray*}
\mathbb{E}[e^{-sW}]&=&\int_{0}^{\infty}\left[\beta \int_{0}^{v}\frac{\mathbb{P}(W\geq v-w)}{\sqrt{w}}se^{-sv}dv\right]\\
&=&\int_{0}^{\infty}\frac{1}{\sqrt{w}}\left[\beta\int_{w}^{\infty}\mathbb{P}(W\geq v-w)se^{-sv}dv\right]dw\\
&=&\beta\int_{0}^{\infty}\frac{e^{-sw}}{\sqrt{w}}dw.\left[1-\mathbb{E}[e^{-sW}]\right],
\end{eqnarray*}
and our claim about the Laplace transform of $W$ follows, because up to a change of variable $(v^2=2sw)$, we have
\begin{equation*}
\int_{0}^{\infty}\frac{e^{-sw}}{\sqrt{w}}dw =\int_{0}^{\infty}\frac{e^{-\frac{1}{2}v^2}}{v/\sqrt{s}}\frac{2v}{2s}dv\\
=\frac{\sqrt{\pi}}{\sqrt{s}}.
\end{equation*}
Hence, as a consequence of the previous computations, we end up with
$\mathbb{E}[e^{-sW}]=\frac{1}{1+c_{\beta}\sqrt{s}}.$
Then $W$ has the same Laplace transform of $c_{\beta}^{2}\dfrac{\mathcal{E}^2}{\mathcal{N}^2}$.
\end{proof}
\begin{proof}[Proof of theorem \ref{theorem2.2}.]
 According to Lemma \ref{lemma 3.4}, the family of distributions of $X_{k}$ is tight. By Lemmas \ref{lemma3.4.2}, \ref{lemma3.4.3} and \ref{lemma3.4.4}, the law of $c_{\beta}\dfrac{\mathcal{E}}{|\mathcal{N}|}$ is the only possible accumulation point of the family of distributions of  $\left(\nu(C_{k}(x))\sqrt{\tau_{e^{-k}}}\mid C_{k}(x)\right)_{k\geq0}.$
 Let $\mathbb{P}$ be a probability measure absolutely continuous with respect to $\mu$, with density $h$. Set $H(x):=\sum_{l\in\mathbb{Z}}h(x,l)$.\\
 Note that by $\mathbb{Z}$-periodicity, the distribution of $\tau_{\epsilon}$ under $\mathbb{P}$ is the same of that under the probability measure with density $(x,l)\mapsto H(x)$ with respect to $\nu\otimes\delta_{0}$.\\
 Assume first that the density $H$ is continuous. Denote by $A_{k}:=\{y:(\nu(C_{k}(y)\sqrt{\tau_{e^{-k}}(y)}>t\}$, then we have:
 \begin{eqnarray*}
 \mathbb{P}(A_{k}|C_{k}(x))&\underset{k\rightarrow\infty}{\sim}&\nu(A_{k}>t|C_{k}(x))\\
 &\underset{k\rightarrow\infty}{\sim}&\mathbb{P}\left(c_{\beta}\frac{\mathcal{E}}{|\mathcal{N}|}\right).
 \end{eqnarray*}
 And so, by the dominated Lebesgue theorem, we get:
 \begin{eqnarray*}
 \mathbb{P}(A_{k})&=&\int\mathbb{P}(A_{k}|C_{k}(x))H(x)d\nu(x)\\
 &\sim&\mathbb{P}\left(c_{\beta}\frac{\mathcal{E}}{|\mathcal{N}|}\right)).
 \end{eqnarray*}
  Now, take in general the density $H$ in $\mathbb{L}^{1}(\nu)$. We use the fact that the set of the continuous functions is dense in $\mathbb{L}^{1}(\nu)$, so that there exists $H_{n}$ continuous such that $H_{n}\overset{\mathbb{L}^{1}(\nu)}{\longrightarrow}H$.
 \begin{eqnarray*}
 \mathbb{P}(A_{k})&=&\int\mathds{1}_{A_{k}}(x)H(x)d\nu(x)\\
 &\leq&\int \mathds{1}_{A_{k}}(x)H_{n}(x)d\nu(x) + ||H_{n}-H||_{\mathbb{L}^{1}(\nu)}.
 \end{eqnarray*}
 We know that there is $n$ such that $\forall$ $\epsilon>0$ $||H_{n}-H||_{\mathbb{L}^{1}(\nu)}<\frac{\epsilon}{2}$. Moreover $H_{n}$ is continuous, then there is $k$ such that $\forall$ $\epsilon>0$ $\big|\int \mathds{1}_{A_{k}}(x)H_{n}(x)d\nu(x)-\mathbb{P}\left(c_{\beta}\frac{\mathcal{E}}{|\mathcal{N}|}\right)\big|<\frac{\epsilon}{2}$. Hence the conclusion follows.
 
\end{proof}
\begin{proof}[Proof of Corollary \ref{corollary}.]
Let us set:
\begin{equation*}
Y_{k}:=\frac{\log{\sqrt{\tau_{e^{-k}}}(.)}-kd}{\sqrt{k}}.
\end{equation*}
We have the case that $\nu$ is a Gibbs measure with a non degenerate H\"{o}lder potential $h$. There is a constant $c_{h}>0$ such that $\log\nu(C_{k}(x))=\sum_{j=-k}^{k}h\circ\sigma^{j}(x)$. This Birkhoff sum follows a central limit theorem (e.g. \cite{bowen}), which implies that:
\begin{equation*}
\frac{\log\nu(C_{k}(.))+kd}{\sqrt{k}}\overset{dist}{\longrightarrow}\mathcal{N}(0,2\sigma^{2}_{h}).
\end{equation*}
Observe that $Y_{k}$ has the following decomposition:
\begin{equation*}
Y_{k}=\frac{\log\left(\nu(C_{k}(.))\sqrt{\tau_{e^{-k}}(.)}\right)}{\sqrt{k}}-\frac{\log\nu(C_{k}(.))+kd}{\sqrt{k}}.
\end{equation*}
Hence, it will be enough to prove that the first term of $Y_{k}$ converges in distribution to 0, which is true due to Theorem \ref{theorem2.2}.
\end{proof}

\section*{Acknowledgment}:I would like to thank my supervisors F. P\`{e}ne and B. Saussol for their help and advice during this work and their availability for answering all my questions.

\protect\bibliographystyle{abbrv}
\protect\bibliographystyle{alpha}
\bibliography{SamWehNicMer}

\end{document}